\documentclass[11pt, a4paper]{amsart}

\usepackage{amssymb, array, amsmath, amscd, pdfpages, enumerate, amsthm, fixltx2e, setspace, hyperref}

\DeclareMathOperator*{\Fix}{Fix}

\DeclareMathOperator*{\R}{Re}
\DeclareMathOperator*{\I}{Im}

\newcommand{\e}{\mathrm{e}}
\newcommand{\ii}{\mathrm{i}}
\newcommand{\dd}{\mathrm{d}}
\newcommand{\Ls}{L^1_\mathrm{s}}
\newcommand{\B}{\mathcal{B}}

\newcommand{\F}{\mathcal{F}}

\newcommand{\RR}{\mathbb{R}}
\newcommand{\CC}{\mathbb{C}}

\newtheorem{thm}{Theorem}[section]
\newtheorem{prp}[thm]{Proposition}
\newtheorem{lem}[thm]{Lemma}
\newtheorem{cor}[thm]{Corollary}
\theoremstyle{definition}
\newtheorem{rem}[thm]{Remark}

\numberwithin{equation}{section}

\begin{document}

\title{A Katznelson-Tzafriri theorem for measures}

\author{David Seifert}
\address{St John's College, St Giles, Oxford\;\;OX1 3JP, United Kingdom}
\email{david.seifert@sjc.ox.ac.uk}

\begin{abstract}
 This article generalises the well-known Katznelson-Tzafriri theorem for a $C_0$-semigroup $T$ on a Banach space $X$, by removing the assumption that a certain measure  in the original result be absolutely continuous. In an important special case the rate of decay is quantified in terms of the growth of the resolvent of the  generator of $T$. These results are closely related to ones obtained recently in the Hilbert space setting by Batty, Chill and Tomilov in \cite{BCT}. The main new idea is to incorporate an assumption on the non-analytic growth bound $\zeta(T)$  which  is equivalent to the assumption made in \cite{BCT} if $X$ is a Hilbert space.  
 \end{abstract}

\thanks{This work was completed with financial support from the EPSRC}

\maketitle

\section{Introduction}\label{sec:intro}

The Katznelson-Tzafriri theorem  is a cornerstone of the asymptotic theory of operator semigroups. Given  a function $a\in L^1(\RR)$, define its Fourier transform $\F a$  by 
\begin{equation}\label{FT}(\F a)( s)=\int_{\RR}\e^{-\ii s t}a(t)\,\dd t 
\end{equation}
 for $s\in \RR$. Further, given a bounded $C_0$-semigroup $T$ on a complex Banach space $X$ and a function $a\in L^1(\RR_+)$, define the operator $\widehat{a}(T)$ for $x\in X$ by 
 \begin{equation}\label{op}
 \widehat{a}(T)x=\int_{\RR_+}a(t)T(t)x\,\dd t.
 \end{equation}
   The  version of the Katznelson-Tzafriri theorem that is of interest here can now be stated as follows; see \cite{ESZ}, \cite{Vu} for proofs.

\begin{thm}\label{KT_cont}
Let $X$ be a complex Banach space and let $T$ be a bounded $C_0$-semigroup on $X$ with generator $A$. Suppose that $\ii\sigma(A)\cap\RR$ is of spectral synthesis and let $a\in L^1(\RR_+)$ be such that $(\F a)(s)=0$ for all $s\in \ii\sigma(A)\cap\RR$. Then $\|T(t)\widehat{a}(T)\|\to0$ as $t\to\infty$.
\end{thm}

 Here a closed  subset $\Lambda$ of $\RR$ is said to be of \textsl{spectral synthesis} if any function $a\in L^1(\RR)$ such that $(\F a)(s)=0$ for all $s\in\Lambda$   can be approximated in $L^1$-norm by elements of  $L^1(\RR)$ whose Fourier transforms vanish in an open neighbourhood of $\Lambda$. In the case where the underlying space is a Hilbert space, the following generalisation of Theorem~\ref{KT_cont} was recently obtained in \cite[Theorem~6.14]{BCT}.
 
 \begin{thm}\label{BCT}
Let $X$ be a complex Hilbert space and let $T$ be a bounded $C_0$-semigroup on $X$ with generator $A$. Suppose that $\ii\sigma(A)\cap\RR$ is of spectral synthesis and that there exists $R>0$ such that $\{\ii s:|s|\geq R\}\subset\rho(A)$ and $\sup\{\|R(\ii s,A)\|:|s|\geq R\}<\infty$. Then, given any  bounded Borel measure $\mu$ on $\RR_+$ such that $(\F\mu)(s)=0$ for all $s\in \ii\sigma(A)\cap\RR$,  $\|T(t)\widehat{\mu}(T)\|\to0$ as $t\to\infty$.
\end{thm}
 
 Here the Fourier transform $\F\mu$ and the operator $\widehat{\mu}(T)$ are defined, for any bounded Borel measure on $\RR$ and $\RR_+$, respectively, and for any bounded $C_0$-semigroup $T$, by formulas analogous to \eqref{FT} and \eqref{op}, with the measure $a(t)\dd t$ replaced by $\dd\mu(t)$ in each case. Thus Theorem~\ref{BCT} generalises Theorem~\ref{KT_cont} in the Hilbert space setting by no longer requiring the measure $\mu$ to be absolutely continuous with respect to Lebesgue measure. 
 
\begin{rem}
It is shown in \cite[Theorem~3.1]{Seifert} that the assumption of spectral synthesis can be dropped in Theorem~\ref{KT_cont} when $X$ is a Hilbert space.  It remains open whether this is true also in the setting of Theorem~\ref{BCT}.
\end{rem} 
 
The purpose of this article is to extend Theorem~\ref{BCT} to Banach spaces, and this is achieved in Theorem~\ref{KT_zeta} below. As it turns out, the condition in Theorem~\ref{BCT} on the resolvent operator  needs to be replaced in the Banach space case by a condition involving the so-called \textsl{non-analytic growth bound} $\zeta(T)$ of the $C_0$-semigroup $T$. This condition requires the semigroup to be close in a certain asymptotic sense to an analytic operator-valued function, and consequently a $C_0$-semigroup $T$ satisfying this condition will be said to be \textsl{asymptotically analytic}. When $X$ is a Hilbert space, this property  is equivalent to the  condition on the resolvent operator in Theorem~\ref{BCT}. The main new idea in this paper is to use a characterisation of analytic analyticity obtained in \cite{BBS}, which makes it possible  to extend Theorem~\ref{BCT} to the Banach space setting under this assumption.
 
 A particularly important instance of Theorem~\ref{BCT} arises when $\ii\sigma(A)\cap\RR=\{0\}$. In this case, the result shows that 
 \begin{equation}\label{decay}
\lim_{t\to\infty}\|T(t)AR(1,A)\|=0
\end{equation}
provided $\{\|R(\ii,s,A)\|:|s|\geq 1\}$ is bounded, as can be seen by choosing $\mu=e-\delta_0$. Here $\delta_0$ denotes the Dirac mass at 0 and the function $e$ is defined, for $t\geq0$, by $e(t)=\e^{-t}$. This result, which provides the main motivation both for the more general  Theorem~\ref{BCT} and for the present paper, is of especial interest in connection with the abstract Cauchy problem  associated with  $T$, namely
\begin{equation}\label{ACP}
\begin{cases}
\dot{u}(t)=Au(t), &\text{$t\geq0$,}\\
u(0)=x,
\end{cases}
\end{equation} 
where $x\in X$. The solution of \eqref{ACP} is given by $u(t)=T(t)x$ for all $t\geq0$. For general $x\in X$, this function $u:\RR_+\to X$ solves \eqref{ACP} only in the  mild sense, however if $x\in D(A)$ then $u$ is a classical solution of \eqref{ACP}. In particular, $u$ is then differentiable with derivative $\dot{u}(t)=T(t)Ax$ for $t\geq0$. Since the resolvent operator $R(1,A)$ maps $X$ onto $D(A)$, \eqref{decay} is simply another way of saying that classical solutions of \eqref{ACP} have derivatives which decay to zero uniformly for all $x\in D(A)$ with graph norm $\|x\|+\|Ax\|\leq1$. Examples of Cauchy problems in which $\ii\sigma(A)\cap\RR=\{0\}$ arise  naturally.  Consider for instance the problem $$\begin{aligned}
\frac{\partial^2 u}{\partial t^2}-\Delta u&=0,\qquad x\in\Omega,\; t>0,\\
\frac{\partial u}{\partial n}+\int_{0}^ta(t-s)\frac{\partial u}{\partial t} (x,s)\,\dd s&=0, \qquad x\in\partial\Omega, \;t>0,
\end{aligned}$$
where $\Omega$ is a bounded open subset of $\RR^n$ for some $n\geq1$ with sufficiently smooth boundary, $\frac{\partial u}{\partial n}$ denotes the outward normal derivative and $a$ denotes a sufficiently well-behaved function defined on $\RR_+$. If $u$ is interpreted as acoustic pressure, the equation can be understood as modelling the evolution of sound in a compressible medium with viscoelastic surface. It is shown in \cite{DFMP}  that the associated semigroup generator $A$ corresponding to this problem has boundary spectrum $\sigma(A)\cap\ii\RR\subset\{0\}$ and is in general non-empty; see also \cite{ABPZ}, \cite{AvTri13} and  \cite{DFMP2}. As is shown in  \cite[Theorem~6.15]{BCT}, it is possible to obtain estimates on the rate of decay in \eqref{decay} when $X$ is a Hilbert space. Theorem~\ref{KT_zeta_rates} below extends this result to the Banach space setting for the class of  asymptotically analytic semigroups. For further background material on the problem at hand, including its connection with a closely related problem having applications to damped wave equations, see \cite[Section~6]{BCT}.
 
 The new results in this paper are contained in Section~\ref{sec:main} below. First, Section~\ref{sec:prel} provides some context for these main results by providing an overview of some of the results obtained in \cite{BCT} concerning the rate of decay in the case where $\ii\sigma(A)\cap\RR=\{0\}.$ The paper concludes with some remarks and open questions, which are collected in Section~\ref{sec:rem}. The notation used throughout  is standard except where introduced explicitly. In particular, given a closed linear operator $A$ on a complex Banach space $X$, the domain of $A$ is denoted by $D(A)$, the spectrum and resolvent sets of $A$ are denoted by $\sigma(A)$ and $\rho(A)$, respectively, and for $\lambda\in\rho(A)$ the resolvent operator is written as $R(\lambda,A)=(\lambda-A)^{-1}$.

 \section{Background on the case of one-point boundary spectrum}\label{sec:prel}
 
Let $T$ be a bounded $C_0$-semigroup with generator $A$ on a complex Banach space $X$ and suppose that $\ii\sigma(A)\cap\RR=\{0\}$. This section collects together several results from \cite{BCT} which relate the rate of decay of $\|T(t)AR(1,A)\|$ as $t\to\infty$ to the behaviour of $\|R(\ii s,A)\|$ as $|s|\to0$. In order to make the relationship precise, it is convenient to have in place two pieces of non-standard notation. Thus a decreasing function $m:(0,1]\to(0,\infty)$   such that  $\|R(\ii s,A)\|\leq m(|s|)$ for $0<|s|\leq1$ will be said to be a \textsl{dominating function (for the resolvent of $A$)}. Likewise a decreasing function  $\omega:\RR_+\to(0,\infty)$  such that $\|T(t)AR(1,A)\|\leq\omega(t)$ for all $t\geq0$ will be said to be a \textsl{dominating function (for  $T$)}. The \textsl{minimal} dominating functions are given, for $s\in(0,1]$ and $t\geq0$, by 
\begin{equation}\label{min_df_cont}\begin{aligned}
m(s)&=\sup\big\{\|R(\ii r,A)\|:s\leq|r|\leq1\big\},\\  
\omega(t)&=\sup\big\{\|T(s)AR(1,A)\|:s\geq t\big\},
\end{aligned}\end{equation} 
respectively.  The function $m$ defined in \eqref{min_df_cont} is continuous and in what follows the same will be assumed to be true of any dominating function $m$. In particular, any such dominating function $m$ possesses a right-inverse $m^{-1}$ defined on the range of $m$. On the other hand, given a dominating function $\omega$ for $T$ such that $\omega(t)\to0$ as $t\to\infty$, let the function $\omega^*:(0,\infty)\to\RR_+$ be given by \begin{equation}\label{w_inv_cont}\omega^*(s)=\min\big\{t\geq0:\omega(t)\leq s\big\}.\end{equation} Then $\omega(\omega^*(s))\leq s$ for all $s>0$, with equality for all $s$ in the range of $\omega$. 

For $T$ and $A$ as above, it follows from the elementary properties of resolvent operators and the fact that $0\in\sigma(A)$ that $m(s)\geq s^{-1}$ for any dominating function $m$ for the resolvent of $A$. As is shown in \cite[Example~6.7]{BCT} by means of a simple direct sum construction,  dominating functions $\omega$ for $T$ need not satisfy any such lower bound and indeed can decay arbitrarily slowly, even for semigroups of normal operators on Hilbert space. On the other hand, the following result shows that cases in which a dominating function $\omega$ for $T$ decays faster than $t^{-1}$ must be precisely of this type involving a direct sum. The proof, which relies on a simple spectral splitting argument combined with a result from \cite{KMSOT}, can be found in \cite[Theorem~6.9]{BCT}.

\begin{thm}\label{max_decay_cont}
Let $X$ be a complex Banach space and let $T$ be a  $C_0$-semigroup on $X$ with generator $A$. Suppose that $\ii\sigma(A)\cap\RR=\{0\}$. Then either $$\limsup_{t\to\infty}t\|T(t)AR(1,A)\|>0$$ or there exist closed $T$-invariant subspaces $X_0$ and $X_1$ of $X$ such that $X_0\subset\Fix(T)$, the restriction $A_1$ of $A$ to $X_1$ is invertible, and $X=X_0\oplus X_1$.
\end{thm}

The following result provides an estimate on the size of $\|R(\ii s,A)\|$ for small and large values of $|s|$; see \cite[Theorem~6.10]{BCT} for a proof.

\begin{thm}\label{res_bd}
Let $X$ be a complex Banach space and let $T$ be a bounded $C_0$-semigroup on $X$ with generator $A$. Suppose $\omega$ is a dominating function for $T$ such that $\omega(t)\to0$ as $t\to\infty$, and let $\omega^*$ be defined as in \eqref{w_inv_cont}. Then $\ii\sigma(A)\cap\RR\subset\{0\}$, $\sup\{\|R(\ii s,A)\|:|s|\geq1\}<\infty$ and, for any $c\in (0,1)$, $$\|R(\ii s,A)\|=O\left(\frac{1}{|s|}+\omega^*(cs)\right)$$ as $s\to0$.
\end{thm}

This result in turn leads to the following bound on how fast the quantity $\|T(t)AR(1,A)\|$ can decay as $t\to\infty$; see \cite[Corollary~6.11]{BCT}.

\begin{cor}\label{cor:lb_cont}
Let $X$ be a complex Banach space and let $T$ be a bounded $C_0$-semigroup on $X$ with generator $A$. Suppose that $\ii\sigma(A)\cap\RR=\{0\}$ and that   \begin{equation}\label{L_cont}\lim_{s\to0}\max\big\{\|s R(\ii s,A)\|,\|s R(-\ii s,A)\|\big\}=\infty,\end{equation} and let  $m$ be the minimal dominating function for the resolvent of $A$ defined in \eqref{min_df_cont}. Then, given any right-inverse $m^{-1}$ of $m$, there exist constants $c,C>0$  such that \begin{equation*}\label{KT_lb_cont}\|T(t)AR(1,A)\|\geq c m^{-1}(Ct)\end{equation*} for all sufficiently large $t\geq0$.
\end{cor}

\begin{rem}
A similar argument to that used in \cite[Theorem~6.10]{BCT} shows that, given any constant $K>M$, where $M$ is as above, there exists $c\in(0,1)$ such that \begin{equation*}\label{eq_cont}\|R(\ii s,A)\|\leq K\left(\frac{1}{|s|}+\omega^*(c|s|)\right)\end{equation*} whenever $|s|$ is sufficiently small. Thus the conclusion \eqref{cor:lb_cont} in fact remains true if \eqref{L_cont} is replaced by the weaker condition that $L>M$, where $M$ is as  above  and $$L=\liminf_{s\to0}\max\big\{\|s R(\ii s,A)\|,\|s R(-\ii s,A)\|\big\}.$$ Taking $A=0$  shows that the conclusion can become false when $L=M$; see also \cite[Remark~6.12]{BCT}
\end{rem}

The next result shows that  it is not necessarily possible even for semigroups of normal operators on a Hilbert space to obtain a corresponding upper bound in terms of $m^{-1}$ for the quantity $\|T(t)AR(1,A)\|$ as $t\to\infty$; see \cite[Theorem~6.13]{BCT} for a proof.

\begin{thm}\label{normal_cont} Let $X$ be a complex Hilbert space and let $T$ be a bounded $C_0$-semigroup of normal operators on $X$ with generator $A$. Suppose  that $\ii\sigma(A)\cap\RR=\{0\}$ and that $\sup\{\|R(\ii s,A)\|:|s|\geq1\}<\infty$. Furthermore, let $m$ be the minimal dominating function for the resolvent of $A$ and let $m^{-1}$ be any right-inverse of $m$. Then, given any constant $c>0$,  $$\|T(t)AR(1,A)\|\leq O(m^{-1}(ct))$$ as $t\to\infty$ if and only if there exists a constant $C>0$ such that $$\frac{m(s)}{m(t)}\geq c\log\left(\frac{t}{s}\right)-C$$ for $0< s,t\leq1$.
\end{thm}

\begin{rem}
For analogues of the above results in the discrete setting of the classical Katznelson-Tzafriri theorem see \cite{Seifert2}.
\end{rem}

\section{Main results}\label{sec:main}

The objective now is to establish an upper bound on  $\|T(t)AR(1,A)\|$ as $t\to\infty$, where $T$ is a bounded $C_0$-semigroup on a complex Banach space $X$ whose generator $A$ satisfies $\ii \sigma(A)\cap\RR=\{0\}$. A result of this type will be obtained in Theorem~\ref{KT_zeta_rates} below, providing a counterpart to Corollary~\ref{cor:lb_cont}. As in \cite{BCT}, Theorem~\ref{KT_zeta_rates} will follow from a more general result by refining its proof in the special case at hand. This more general result is Theorem~\ref{KT_zeta} below, which is analogous to Theorem~\ref{BCT} but does not require $X$ to be a Hilbert space. On the other hand, the assumption on the resolvent operator that appears in the statement of Theorem~\ref{BCT} needs to be modified, and this requires some further notation.

Given a complex Banach space $X$ and a  set $\Omega\subset\CC$, a function $S:\Omega\to\B(X)$ will be said to be \textsl{exponentially bounded} if there exist constants $C\geq0$ and $\omega\in\RR$ such that $\|S(\lambda)\|\leq C\e^{\omega |\lambda|}$ for all $\lambda\in \Omega$. The space of all exponentially bounded holomorphic $\B(X)$-valued functions on $\Omega$ will be denoted by $H(\Omega;\B(X))$. Further, given any exponentially bounded function $S$  defined on $(0,\infty)$, let its \textsl{growth bound} $\omega_0(S)$ be given by $$\omega_0(S)=\inf\big\{\omega\in\RR: \|S(t)\|\leq M\e^{\omega t}\;\mbox{for some $M\geq1$ and all $t>0$}\big\}.$$ The \textsl{non-analytic growth bound} $\zeta(T)$ of a $C_0$-semigroup $T$ is defined as 
\begin{equation}\label{zeta_def}
\zeta(T)=\inf\big\{\omega_0(T-S):S\in H(\Sigma_\theta;\B(X))\;\mbox{for some $\theta>0$}\big\},
\end{equation}
 where $\Sigma_\theta=\{\lambda\in\CC\backslash\{0\}:|\arg(\lambda)|<\theta\}$. Thus the non-analytic growth bound measures the degree to which $T$ can, or rather \textsl{cannot}, be approximated asymptotically by  exponentially bounded analytic functions defined on certain sectors. It is clear that $\zeta(T)\leq \omega_0(T)$ for any $C_0$-semigroup $T$ and that $\zeta(T)=-\infty$ when $T$ is analytic. It is shown in \cite[Theorem~5.7]{BBS} that $\zeta(T)=-\infty$ also if $T$ is eventually differentiable or if $T$ has an $L^p$-resolvent for some $p\in(1,\infty)$, in the sense that there exist $\alpha\in\RR$ and $\beta\geq0$ such that $\{\lambda\in\CC:\R\lambda\geq\alpha, |\I\lambda|\geq\beta\}\subset\rho(A)$ and $$\int_{|s|\geq \beta}\|R(\alpha+\ii s,A)\|^p\,\dd s<\infty.$$ Furthermore, if $$\omega_{\mathrm{ess}}(T)=\inf\big\{\omega\in\RR: \|T(t)\|_{\mathrm{ess}}\leq M\e^{\omega t}\;\mbox{for some $M\geq1$ and all $t\geq0$}\big\}$$
  denotes the \textsl{essential growth bound} of $T$, where for $Q\in \B(X)$ $$\|Q\|_{\mathrm{ess}}=\inf\{\|Q-K\|:K\in\B(X)\;\mbox{is compact}\},$$then $\zeta(T)\leq\omega_{\mathrm{ess}}(T)$; see \cite[Proposition~5.3]{BBS}. A $C_0$-semigroup $T$ will be said to be \textsl{asymptotically analytic} if $\zeta(T)<0$. 
 
 Given $\alpha\in\RR$ and $\beta\geq0$, let $Q_{\alpha,\beta}=\{\lambda\in\CC:\R\lambda\geq\alpha, |\I\lambda|\geq\beta\}$ and, for any semigroup generator $A$, let $$\begin{aligned}s_0^\infty(A)=\inf\big\{\alpha\in\RR: & \;Q_{\alpha,\beta}\subset\rho(A)\;\mbox{and $\|R(\lambda,A)\|$ is uniformly}\\& \;\mbox{bounded on $Q_{\alpha,\beta}$ for some $\beta\geq0$}\big\}.\end{aligned}$$ It is shown in \cite[Proposition~2.4]{BBS} that $s_0^\infty(A)\leq \zeta(T)$. When $X$ is a Hilbert space, the following non-analytic analogue of the Gearhart-Pr\"uss theorem holds; see  \cite[Example~3.12]{BS}.

\begin{thm}\label{NA_GP} 
Let $X$ be a complex Hilbert space and let $T$ be a $C_0$-semigroup on $X$ with generator $A$. Then $s_0^\infty(A)= \zeta(T)$. 
\end{thm}

The proof of this result relies on Plancherel's theorem and the equivalence of conditions (i) and (iii) in  the following result, which will be crucial in what follows; see \cite[Theorem~3.6]{BS} for a proof. Here $\Ls(\RR;\B(X))$ denotes the space of maps $S:\RR\to\B(X)$ such that $t\mapsto S(t)x$ is Bochner measurable for all $x\in X$ and such that there exists $g\in L^1(\RR)$ with  $\|S(t)\|\leq g(t)$ for almost all $t\in \RR$. Furthermore, given a semigroup generator $A$ satisfying $s_0^\infty(A)<0$, a smooth function $\phi\in C^\infty(\RR)$ will be said to be \textsl{compatible with A} if there exists a bounded open interval $I\subset\RR$ satisfying $\{s\in \RR:\ii s\in\sigma(A)\}\subset I$ and if $\phi(s)=0$ for all $s\in I$ and $\phi(s)=1$ when $|s|$ is sufficiently large.

\begin{thm}\label{zeta_FM} Let $X$ be a complex Banach space and let $T$ be a $C_0$-semigroup on $X$ with generator $A$. Suppose that $1\leq p<\infty$. Then the following are equivalent:

\begin{enumerate}
\item[(i)] $T$ is asymptotically analytic;
\item[(ii)] $s_0^\infty(A)<0$ and, for any compatible function $\phi\in C^\infty(\RR)$, there exists a map $S\in\Ls(\RR;\B(X))$ such that $\phi(s)R(\ii s,A)=(\mathcal{F}S)(s)$ for all $s\in\RR$;
\item[(iii)] $s_0^\infty(A)<0$ and, for any compatible function $\phi\in C^\infty(\RR)$, the map $s\mapsto\phi(s)R(\ii s,A)$ is a Fourier multiplier on $L^p(\RR;X)$.
\end{enumerate}
\end{thm}

 Here  $\phi(s)R(\ii s,A)$ is taken to be zero whenever $\phi(s)=0$ and the Fourier transform $\mathcal{F}$ is taken in the strong sense, so that, given $S\in\Ls(\RR;\B(X))$ and $s\in \RR$, $$(\mathcal{F}S)(s)x=\int_\RR\e^{-\ii st} S(t)x\,\dd t$$ for all $x\in X$. Note also that, even though the maps $s\mapsto\phi(s)R(\ii s,A)$ depend on the choice of $\phi$, the property of such a map being a Fourier multiplier on $L^p(\RR;X)$ is independent of this choice; see \cite[Remark~2.2]{BS}.

It is now possible to extend Theorem~\ref{BCT} to the Banach space setting, as is done in Theorem~\ref{KT_zeta} below. The proof of this result combines the method used in \cite{BCT} with Theorem~\ref{zeta_FM} and the following simple lemma, which is probably well known. Here $\mathcal{S}(\RR)$ denotes the  space of Schwartz functions on $\RR$.

\begin{lem}\label{distr_lem}
Let $X$ be a complex Banach space and let $S\in L^1(\RR;X)$.
\begin{enumerate}
\item[(a)] If $\mu$ is a bounded Borel measure on $\RR$, then  $\mu*S\in L^1(\RR;X)$. 
\item[(b)] If \begin{equation}\label{distr_eq}\int_\RR\rho(t)S(t)\,\dd t=0\end{equation} for all $\rho\in\mathcal{S}(\RR)$, then $S(t)=0$ for almost all $t\in\RR$.
\end{enumerate}
\end{lem} 

\begin{proof}[\textsc{Proof}] 
Part (a) is a special case of  the vector-valued version of Fubini's theorem; see  for instance \cite[Theorem~1.1.9]{Green_book} and \cite[Chapter~III, Section~11, Theorem~9]{DS58}. For a direct proof, let the map $F:\RR^2\to X$ be given by $F(s,t)=S(t-s)$ and note that, by Pettis' theorem (see \cite[Theorem~1.1.1]{Green_book}), there is no loss of generality in assuming that $X$ is separable. Moreover, the map $F$ is measurable with respect to the product measure of $\mu$ and the Lebesgue measure on $\RR$, as can be seen for instance by another application of Pettis' theorem and a simple approximation argument based on the fact that, given any Lebesgue measurable subset $E$ of $\RR$, there exists a Borel measurable set $E'\subset\RR$ such that the symmetric difference  $E\triangle E'$ is null. Since \begin{equation}\label{Fub_bd}\int_\RR\int_\RR \|F(s,t)\|\,\dd t\,\dd|\mu|(s)<\infty,\end{equation}  it follows form the scalar-valued versions of Tonelli's and Fubini's theorems that $(s,t)\mapsto\|F(s,t)\|$ is integrable over $\RR^2$ with respect to the product measure and that $\int_\RR\|F(s,t)\|\,\dd\mu(s)$ exists for almost all $t\in \RR$. Hence so does $\int_\RR F(s,t)\,\dd\mu(s)$. Furthermore,  for each $\phi\in X^*$, the map $$t\mapsto\int_\RR \phi(F(s,t))\,\dd\mu(s)=\phi\left(\int_\RR F(s,t)\,\dd\mu(s)\right)$$ is measurable on $\RR$.  By  Pettis' theorem, the map  $\mu*S:t\mapsto\int_\RR F(s,t)\,\dd\mu(s)$ is Bochner measurable on $\RR$, and the result now follows, since $$\int_\RR\|(\mu*S)(t)\|\,\dd t\leq\int_\RR\int_\RR \|F(s,t)\|\,\dd|\mu|(s)\,\dd t<\infty$$ by the scalar-valued version of Fubini's theorem and \eqref{Fub_bd}.

 The proof of (b) runs along similar lines. Once again, by Pettis' theorem, it is possible to assume that $X$ is separable. Let $\{x_n:n\geq1\}$ be a dense subset of $X$ and, for each $n\geq1$, let $\phi_n\in X^*$ be such that $\|\phi_n\|=1$ and $\phi_n(x_n)=\|x_n\|$. The existence of such functionals is a consequence of the Hahn-Banach theorem. Applying $\phi_n$ to both sides of \eqref{distr_eq} shows that $$\int_\RR \rho(t)\phi_n(S(t))\,\dd t=0$$ for all $\rho\in\mathcal{S}(\RR)$, and hence, for each $n\geq1$, there exists a null subset $E_n$ of $\RR$ such that $\phi_n(S(t))=0$ for all $t\in\RR\backslash E_n$. Let $E=\bigcup_{n\geq1} E_n$, noting that $E$ itself is null, and suppose that $t\in\RR\backslash E$. Given $\varepsilon>0$,  there exists $n\geq1$ such that  $\|x_n-S(t)\|<\varepsilon/2$. Thus $$\|S(t)\|\leq \|x_n\|+\|x_n-S(t)\|<\phi_n(x_n-S(t))+\frac{\varepsilon}{2}<\varepsilon.$$ Since $\varepsilon>0$ was arbitrary, it follows that $S(t)=0$ for all $t\in\RR\backslash E$.
\end{proof}

The next result, then, is an unquantified version of the Katznelson-Tzafriri theorem for measures which generalises Theorem~\ref{BCT} to the Banach space setting.  

\begin{thm}\label{KT_zeta}
Let $X$ be a complex Banach space and let $T$ be a bounded asymptotically analytic $C_0$-semigroup $X$ with generator $A$. Suppose  that $\ii\sigma(A)\cap\mathbb{R}$ is of spectral synthesis and let $\mu$ be any bounded Borel measure on $\RR_+$ such that $(\F\mu)(s)=0$ for all $s\in\ii\sigma(A)\cap \RR$. Then $\|T(t)\widehat{\mu}(T)\|\to0$ as $t\to\infty$.
\end{thm}

\begin{proof}
Since $s_0^\infty(A)\leq\zeta(T)$, the assumption of asymptotic analyticity implies that $\{\ii s:|s|\geq R\}\subset\rho(A)$ and that  $\sup\{\|R(\ii s,A)\|:|s|\geq R\}<\infty$  for some sufficiently large $R>0$. Fix a function $\varphi\in C_\mathrm{c}^\infty(\mathbb{R})$ such that $\varphi(s)=1$ for $|s|\leq R$ and let  $\psi=\F^{-1}\varphi$. Then $\psi\in \mathcal{S}(\RR)$, since $C_\mathrm{c}^\infty(\mathbb{R})\subset\mathcal{S}(\RR)$ and $\mathcal{F}$ maps $\mathcal{S}(\RR)$ bijectively onto itself. Furthermore, let the bounded Borel measures $\nu$ and $\xi$ on $\mathbb{R}$ be defined by $\nu=\mu*\psi$ and $\xi=\mu*(\delta_0-\psi)$, so that $\mu=\nu+\xi$, and let the functions $F,G:\mathbb{R}\rightarrow \B(X)$ be given by 
\begin{equation}\label{F,G}
F(t)=\int_\mathbb{R} T(s+t)\,\dd\nu(s)\qquad\mbox{and}\qquad G(t)=\int_\mathbb{R} T(s+t)\,\dd\xi(s).
\end{equation}
 Here the semigroup has been extended by zero on $(-\infty,0)$ and the integrals are to be understood in the strong sense. It is clear that $T(t)\widehat{\mu}(T)=F(t)+G(t)$ for all $t\geq0$, and hence the result will follow once it has been established that both $\|F(t)\|\to0$ and $\|G(t)\|\to0$ as $t\to\infty$.

Consider first the function $F$. Since the measure  $\nu$ is absolutely continuous with respect to Lebesgue measure on $\RR$, there exists a function $a\in L^1(\RR)$ such that $\dd\nu(t)=a(t)\dd t$, and hence $(\F a)( s)=(\F\mu)( s)\varphi(s)$ for all $s\in \RR$.  In particular, $\F a$ vanishes on $\ii\sigma(A)\cap \mathbb{R}$. By assumption this set is of spectral synthesis, and hence there exist functions $a_n\in L^1(\RR)$, $n\geq1$, such that  $\F a_n$ vanishes in a neighbourhood of  $\ii\sigma(A)\cap \mathbb{R}$ for each $n\geq1$ and $\|a_n-a\|_1\rightarrow0$ as $n\rightarrow\infty$.  Since functions with compactly supported Fourier transform are dense in $L^1(\RR)$, there is no loss of generality in assuming that each $\F a_n$ has compact support. By the dominated convergence theorem and Parseval's identity, \begin{equation*}\label{Parseval}\begin{aligned}\int_\RR a_n(s)T(t+s)\,\dd s &=\lim_{\alpha\rightarrow0+} \int_\RR a_n(s-t)\e^{-\alpha s}T(s)\,\dd s\\&=\lim_{\alpha\rightarrow0+}\frac{1}{2\pi} \int_\RR \e^{\ii st}(\F a_n)(-s)R(\alpha+\ii s,A)\,\dd s\\&=\frac{1}{2\pi} \int_\RR \e^{\ii st}(\F a_n)(-s)R(\ii s,A)\,\dd s
\end{aligned}\end{equation*}  for all $t\in\RR$ and $n\geq1.$ Since the last integral exists as a Bochner integral in $\B(X)$,  $$\left\|\int_\RR a_n(s)T(t+s)\,\dd s\right\|\rightarrow0$$ as $t\rightarrow\infty$ by the Riemann-Lebesgue Lemma. Now $$\left\|F(t)-\int_\RR a_n(s)T(t+s)\,\dd s\right\|\leq M\|a-a_n\|_1$$ for all $t\in \RR$, and hence 
\begin{equation*}\label{F_conv}
\left\|F(t)-\int_\RR a_n(s)T(t+s)\,\dd s\right\|\to0
\end{equation*}
as $n\rightarrow\infty$, uniformly for $t\in \RR$. It follows that $\|F(t)\|\rightarrow0$ as $t\rightarrow\infty$.

Now consider the function $G$. Since the map $\phi\in C^\infty(\RR)$ given by $\phi(s)=1-\varphi(-s)$ is compatible with  $A$, it follows from Theorem~\ref{zeta_FM} that there exists a map $S\in \Ls(\RR;\B(X))$ such that $\phi(s)R(\ii s,A)=(\F S)(s)$ for all $s\in \RR$. Writing $\mu'$ for the bounded Borel measure on $(-\infty,0]$ satisfying $\mu'(E)=\mu(-E)$ for any Borel subset $E$ of $(-\infty,0]$,  part (a) of Lemma~\ref{distr_lem} shows that $\mu'*S\in \Ls(\RR;\B(X)).$ For $\rho\in\mathcal{S}(\RR)$ and $x\in X$, it follows by Fubini's theorem, the dominated convergence theorem and Parseval's identity that $$\begin{aligned}\int_\RR \rho(t)G(t)x\,\dd t &=\int_\RR \rho(t)\int_\RR T(s+t)x\,\dd\xi(s)\,\dd t \\&=
\lim_{\alpha\rightarrow0+} \int_\RR\int_\RR \rho(t-s)\e^{-\alpha t}T(t)x\,\dd t\,\dd\xi(s)\\&=
\lim_{\alpha\rightarrow0+} \int_\RR\int_\RR\e^{\ii st} (\F^{-1}\rho)(t)R(\alpha+\ii t,A)x\,\dd t\,\dd\xi(s)\
\\&=\lim_{\alpha\rightarrow0+} \int_\RR(\F^{-1}\rho)(t)(\F\mu)(-t)(1-\varphi(-t)) R(\alpha+\ii t,A)x\,\dd t
\\&= \int_\RR (\F^{-1}\rho)(t)(\F\mu)(-t)(\F S)(t)x\,\dd t
\\&= \int_\RR \rho(t)S_\mu(t)x\,\dd t,
\end{aligned}$$
where  $S_\mu=\mu'*S$.   Fix $x\in X$ and, for $t\in\RR$, let $S_x\in L^1(\RR;X)$ be given by $S_x(t)=G(t)x-S_\mu(t)x$. Then  $S_x(t)=0$ for almost all $t\in\RR$ by part (b) of Lemma~\ref{distr_lem}, and hence $G(\cdot)x\in L^1(\RR;X)$.  Since the map $\Phi:X\rightarrow L^1(\RR;X)$ given by $\Phi(x)=G(\cdot)x$ has closed graph, it follows from  the closed graph theorem that, for some constant $C>0$, 
\begin{equation}\label{G_bound}
\int_\RR \left\|G(t)x\right\|\dd t\leq C\|x\|
\end{equation}
 for all $x\in X$. Now, given any $x\in X$ and $t\geq0$, $$\begin{aligned}\int_0^t T(t-s)G(s)x\,\dd s&=\int_0^t\left(\int_\mathbb{R}-\int_{[-t,-s)}\right)T(t+r)x\,\dd\xi(r)\dd s\\&=tG(t)x+\int_{[-t,0)}rT(t+r)x\,\dd\xi(r)\end{aligned}$$ and hence, for $t>0$, 
 \begin{equation}\label{G_formula}
 G(t)x=\frac{1}{t}\left(\int_0^t T(t-s)G(s)x\,\dd s-\int_{[-t,0)}rT(t+r)x\,\dd\xi(r)\right).
 \end{equation}
   By \eqref{G_bound}, 
   \begin{equation}\label{G_bd}\begin{aligned}\left\|\int_0^t T(t-s)G(s)\,\dd s\right\|&\leq M C,\end{aligned}
   \end{equation}
    where $M=\sup\{\|T(t)\|:t\geq0\}$, and   the dominated convergence theorem gives $$\left\|\int_{[-t,0)}\frac{r}{t}T(t+r)\,\dd\xi(r)\right\|\rightarrow0$$ as $t\rightarrow\infty$.  Thus $\|G(t)\|\rightarrow0$ as $t\rightarrow\infty$ and the result follows.
\end{proof}

\begin{rem}
As observed in \cite[Remark~3.8]{BS}, the map $S\in \Ls(\RR;\B(X))$ in the above proof can be written down explicitly. Indeed, given $x\in X$,   $$S(t)x=T(t)x-T_R(t)x-\frac{1}{2\pi\ii}\int_{|s|\geq R}\e^{\ii st}\varphi(-s)R(\ii s,A)x\,\dd s$$ for almost all $t\in \RR$. Here $\varphi$ and $R$ are as above, the semigroup $T$ has again been extended by zero on $(-\infty,0)$ and  $$T_R(t)=\frac{1}{2\pi\ii}\int_{\Gamma_R}\e^{\lambda t} R(\lambda, A)\,\dd\lambda$$ for any path $\Gamma_R$ that connects the points $\pm\ii R$ and otherwise lies in $\{\lambda\in\CC:\R\lambda>0\}.$ By Cauchy's theorem, the definition of $T_R$ is independent of the  choice of the contour $\Gamma_R$.  
\end{rem}

As stated near the beginning of this section, there are several important cases in which a semigroup is known to be asymptotically analytic. The first part of the following corollary corresponds to \cite[Theorem~6.14]{BCT}. Note that a $C_0$-semigroup $T$ is said to be \textsl{uniformly exponentially balancing} if there exist $\omega\in\RR$ and a non-zero $P\in\B(X)$ such that $\|\e^{-\omega t}T(t)-P\|\to0$ as $t\to\infty$, and recall that  $T$ is said to be \textsl{quasi-compact} if  $\omega_{\mathrm{ess}}(T)<\omega_0(T)$.

\begin{cor}\label{KT_zeta_cor}
Let $X$ be a complex Banach space and let $T$ be a bounded $C_0$-semigroup $X$ with generator $A$. Suppose that $\ii\sigma(A)\cap\mathbb{R}$ is of spectral synthesis, and let $\mu$ be any bounded Borel measure on $\RR_+$ such that $(\F\mu)(s)=0$ for all $s\in\ii\sigma(A)\cap \RR$. Then $$\lim_{t\to\infty}\|T(t)\widehat{\mu}(T)\|=0$$ provided one of the following conditions is satisfied:
\begin{enumerate}
\item[(i)] $X$ is a Hilbert space and there exists $R>0$ such that $ \{\ii s:|s|\geq R\}\subset\rho(A)$ and $\sup\{\|R(\ii s,A)\|:|s|\geq R\}<\infty$;
\item[(ii)] $T$ is uniformly exponentially balancing;   
\item[(iii)] $T$ is quasi-compact;
\item[(iv)] $T$ has $L^p$-resolvent for some $p\in (1,\infty)$;
\item[(v)] $T$ is eventually differentiable.
\end{enumerate}
\end{cor}

\begin{proof}
It suffices to show that in each of the cases $T$ is asymptotically analytic, since the result will then follow from Theorem~\ref{KT_zeta}. In the first case this is a consequence of Theorem~\ref{NA_GP}, since the assumption on the resolvent implies, by a simple Neumann series argument,  that $s_0^\infty(A)<0$. In the second case the claim follows from the fact that for uniformly exponentially balancing semigroups $\zeta(T)<\omega_0(T)$ (see \cite[Proposition~4.5.9]{DrBlake} and also \cite{Thi98}), while in case (iii) it is a consequence of the estimate $\zeta(T)\leq\omega_{\mathrm{ess}}(T)$. If either (iv) or (v) holds, then $\zeta(T)=-\infty$, as mentioned above.
\end{proof}

The final result of this section is a quantified version of Theorem~\ref{KT_zeta} in the special case where, in the notation of Section~\ref{sec:intro}, $\mu=e-\delta_0$ so that $\widehat{\mu}(T)=AR(1,A)$. This result is a generalisation to the Banach space setting of \cite[Theorem~6.15]{BCT} and provides a counterpart to Corollary~\ref{cor:lb_cont}. It is possible to  formulate versions of the result in which the assumption of asymptotic analyticity is replaced by one of the conditions (i)--(v) of Corollary~\ref{KT_zeta_cor}. 

\begin{thm}\label{KT_zeta_rates}
Let $X$ be a complex Banach space and let $T$ be a bounded asymptotically analytic $C_0$-semigroup $X$ with generator $A$. Suppose that $\ii\sigma(A)\cap\mathbb{R}=\{0\}$. Then $\|T(t)AR(1,A)\|\to0$ as $t\to\infty$.  In fact, given any dominating function $m$ for the resolvent of $A$,  any right-inverse $m^{-1}$ of $m$ and any $\varepsilon\in(0,1)$,  
\begin{equation}\label{rate}
\|T(t)AR(1,A)\|=O\big(m^{-1}(t^{1-\varepsilon})\big)
\end{equation}
 as $t\to\infty$.
\end{thm}

\begin{proof}
Since  the set $\{0\}$ is of spectral synthesis, the first statement follows immediately from Theorem~\ref{KT_zeta} applied to the bounded Borel measure on $\RR_+$ given by $\mu=e-\delta_0$, whose Fourier transform is given by $$(\F{\mu})(s)=-\frac{\ii s}{1+\ii s}.$$ 

The quantified statement follows from a modification of the argument given in the proof of Theorem~\ref{KT_zeta}.  Fix a function $\varphi\in C_\mathrm{c}^\infty(\mathbb{R})$ such that $\varphi(s)=1$ for $|s|\leq2$ and let $\psi\in \mathcal{S}(\RR)$ be given by $\psi=\F^{-1}\varphi$. As in the proof of Theorem~\ref{KT_zeta}, let the bounded Borel measures $\nu$ and $\xi$ on $\mathbb{R}$ be defined by $\nu=\mu*\psi$ and $\xi=\mu*(\delta_0-\psi)$. Then $\mu=\nu+\xi$ and $T(t)\widehat{\mu}(T)=F(t)+G(t)$ for all $t\geq0$, where  $F,G:\mathbb{R}\rightarrow \B(X)$ are as defined in \eqref{F,G}.  As before, the measure $\nu$ is absolutely continuous with respect to Lebesgue measure on $\RR$ and hence there exists a function $a\in L^1(\RR)$ such that $\dd\nu(t)=a(t)\dd t$. In fact,  
\begin{equation}\label{a_formula}
a(t)=\int_0^\infty\psi(t-s)\e^{-s}\,\dd s-\psi(t)
\end{equation}
 for almost all $t\in \RR$. As in the proof of \cite[Theorem~6.15]{BCT}, a refinement of the argument used in the proof of Theorem~\ref{KT_zeta} to show that $\|F(t)\|\to0$ as $t\to\infty$  in fact gives that, for any integer $k\geq1$, \begin{equation}\label{F_bd}\|F(t)\|=O\left(\frac{1}{t^{\varepsilon(k+1)-1}}+m^{-1}(t^{1-\varepsilon})\right)\end{equation} as $t\to \infty$. 
 
 In order to obtain an estimate on $\|G(t)\|$ as $t\to\infty$, note first that, since $\mu$ is supported in $\RR_+$, $\xi$ coincides with  $-\nu$ on $(-\infty,0)$. In particular, letting $M=\sup\{\|T(t)\|:t\geq0\}$, $$\left\|\int_{[-t,0)}rT(t+r)\,\dd\xi(r)\right\|\leq M \int_{(-\infty,0)}|r| \,\dd|\xi|(r)$$ for all $t\geq0$ and hence by \eqref{a_formula} $$\sup_{t\geq0}\left\|\int_{[-t,0)}rT(t+r)\,\dd\xi(r)\right\|\leq M \int_{-\infty}^0|r|\left|\psi(r)-\int_0^\infty \psi(r-s)\e^{-s}\,\dd s\right|\,\dd r<\infty.$$ It follows from \eqref{G_formula} and \eqref{G_bd} that \begin{equation}\label{G_t}\|G(t)\|=O(t^{-1})\end{equation}
as $t\to\infty$.  To conclude the proof, choose $k\geq1$ in \eqref{F_bd} so  that $\varepsilon(k+1)\geq2$.  Since $m^{-1}(t^{1-\varepsilon})\geq t^{\varepsilon-1}\geq t^{-1}$ for $t\geq 1$, it follows from \eqref{F_bd} and \eqref{G_t} that \eqref{rate} holds.
\end{proof}

\begin{rem}
 From the point of view of decay rates  the case $\ii\sigma(A)\cap\RR=\emptyset$, which could be dealt with in a similar way, is of limited interest in Theorem~\ref{KT_zeta_rates}. Indeed,  \cite[Proposition~2.4]{BBS} shows that $\omega_0(T)=\max\{\zeta(T),s(A)\}$, and hence the rate of decay in such cases is necessarily at least exponential if $T$ is asymptotically analytic.
\end{rem}

\section{Concluding remarks}\label{sec:rem}

This final section contains some remarks and open questions. The first concerns the rate of decay obtained in Theorem~\ref{KT_zeta_rates}, the rest relate to the property of asymptotic analyticity of a $C_0$-semigroup $T$. Throughout, $X$ will be a complex Banach space and $T$ a bounded $C_0$-semigroup on $X$ with generator~$A$.

\subsection{Non-optimality of the rate of decay}

Theorem~\ref{normal_cont} shows that one cannot in general hope for an upper bound in terms of $m^{-1}(C t)$ for any constant $C>0$ in \eqref{rate}.  In fact, \cite[Example~6.16]{BCT} shows that the estimate in \eqref{rate} can be optimal when $m(s)$ grows very rapidly as $s\to0+$ but that it is not optimal in general. In the special case where $m(s)=Cs^{-\alpha}$ for some $C>0$ and some $\alpha\geq1$, it is shown by a completely different method in \cite[Theorem~7.6]{BCT} that the right-hand side of \eqref{rate} can be replaced by the optimal $O(t^{-1/\alpha})$ when $X$ is a Hilbert space. Results such as \cite[Theorem~1.5]{BD}, \cite[Theorem~4.1]{BT} and \cite[Proposition~3.1]{Mar} (see also \cite{Seifert2} for analogous results in the discrete setting) suggest that for general Banach spaces it might be possible to obtain an estimate differing only by a logarithmic factor from this polynomial rate of decay rather than having the wrong power as in \eqref{rate}. It remains an open question whether this is indeed the case.

\subsection{Necessity of asymptotic analyticity}

 Theorem~\ref{res_bd} and a simple Neumann series argument show that \eqref{decay} implies  $s_0^\infty(A)<0$. If $X$ is a Hilbert space, it follows from Theorem~\ref{NA_GP} that $T$ is asymptotically analytic whenever \eqref{decay} holds. The same conclusion holds if  $T$ is hyperbolic (see for instance \cite[Section~V.1.c]{EN} for background), since in this case $T$ is asymptotically analytic if and only if $s_0^\infty(A)<0$, by \cite[Proposition~3.5.2]{DrBlake}. It remains unknown whether  \eqref{decay} implies asymptotic analyticity for general  $C_0$-semigroups acting on arbitrary Banach spaces.

This question is related to a further open question which concerns the non-analytic growth bound $\zeta(T)$ and another growth bound associated with $T$. Indeed, allowing the function $S$ in \eqref{zeta_def} to vary over all functions $S:(0,\infty)\to\B(X)$ which are continuous with respect to the norm topology on $\B(X)$ gives the definition of  $\delta(T)$, which is referred to sometimes as the \textsl{critical growth bound}. It is clear from the definitions that $\delta(T)\leq \zeta(T)$.  Furthermore, $s_0^\infty(A)\leq\delta(T)$ so that, if $X$ is a Hilbert space,  $\delta(T)=\zeta(T)=s_0^\infty(A)$ by Theorem~\ref{NA_GP}. These results along with other cases in which $\delta(T)$ and $\zeta(T)$  coincide are contained in \cite[Section~5]{BBS}, but it remains open whether this is always the case. The following result shows that if it were known that $\zeta(T)=\delta(T)$ at least when $\ii\sigma(A)\cap\mathbb{R}=\{0\}$, then  asymptotic analyticity would be not only sufficient but also necessary  for \eqref{decay} to hold in this case. Since $s_0^\infty(A)\leq\delta(T)$, the result strengthens  the above conclusion  derived from  Theorem~\ref{res_bd} that $s_0^\infty(A)<0$.

\begin{prp}\label{delta}
 Let $X$ be a complex Banach space and let $T$ be a bounded $C_0$-semigroup on $X$ with generator $A$. Suppose that  \eqref{decay} holds. Then $\delta(T)<0$. 
\end{prp}

\begin{proof}
Note that, for $t\geq0$, $$T(t)=T(t)R(1,A)-T(t)AR(1,A).$$ The first term on the right-hand side is continuous in $t$ with respect to the norm topology on $\B(X)$. Letting $$D(t)=\limsup_{h\to0+}\|T(t+h)-T(t)\|$$ for $t\geq0$, it follows that $$D(t)\leq \limsup_{h\to0+}\|(T(t+h)-T(t))AR(1,A)\|\leq D(0)\|T(t)AR(1,A)\|,$$ and hence $D(t)\to0$ as $t\to\infty$ by the assumption that \eqref{decay} holds. Since $D$ is submultiplicative by \cite[Proposition~3.5]{Blake}, it follows that $\omega_0(D)<0$. But by \cite[Proposition~5.1]{BBS} $\omega_0(D)=\delta(T)$, which gives the result.
\end{proof}

Combining this observation with  \cite[Corollary~4.5.6]{DrBlake}, it follows that any bounded $C_0$-semigroup $T$ which satisfies \eqref{decay} and whose generator $A$ satisfies $\ii\sigma(A)\cap(\RR+\ii\omega)=\emptyset$ for some $\omega\in(\delta(T),0)$ must be asymptotically analytic. It follows from Theorem~\ref{res_bd} that this is true in particular if $\{0\}$ is an isolated point of $\sigma(A)$.

\subsection{Towards an alternative characterisation of asymptotic analyticity}
 
The following result is proved in \cite[Theorem~3.6]{Blake}; see also \cite[Proposition~4.3]{NaPo00} and \cite{AMM91}. Here, given $\omega\in\RR$, $\Omega_\omega$ denotes the set $\{\lambda\in\CC:|\lambda|>\e^{\omega }\}$.

\begin{thm}\label{SMT}
Let $X$ be a complex Banach space and let $T$ be a $C_0$-semigroup on $X$ with generator $A$. Then $$\sigma(T(t))\cap\Omega_{t\delta(T)}=\exp(t\sigma(A))\cap\Omega_{t\delta(T)}$$ for all $t>0$.
\end{thm}

 Now, letting $$\begin{aligned}s^\infty(A)=\inf\big\{&\alpha\in\RR: Q_{\alpha,\beta}\subset\rho(A)\;\mbox{for some $\beta\geq0$}\big\}, \end{aligned}$$ it is clear that $s^\infty(A)\leq s_0^\infty(A)$ and hence $s^\infty(A)\leq \zeta(T)$. It follows from Theorem~\ref{SMT} that $-1\in\rho(T(t))$ for all sufficiently small $t>0$ if $T$ is asymptotically analytic, since in this case $\delta(T)<0$. The following partial converse is proved in \cite[Sections~4.3 and 4.5]{DrBlake}; see also \cite{BBN01}.

\begin{thm}\label{SIT}
Let $X$ be a complex Banach space and let $T$ be a bounded $C_0$-semigroup on $X$ with generator $A$. Suppose there exists a non-null set $I\subset\RR_+$ such that $-1\in\rho(T(t))$ for all $t\in I$. Then $s_0^\infty(A)<0$. 
\end{thm}

In particular, when combined with Theorem~\ref{NA_GP} this shows that if $T$ is a bounded $C_0$-semigroup and $X$ is a Hilbert space, then $T$ is asymptotically analytic if and only if $-1\in\rho(T(t))$ for all sufficiently small $t>0$. On a general Banach space, no such characterisation is known. Nevertheless, Theorems~\ref{SMT} and \ref{SIT} suggest at least loosely that the property of asymptotic analyticity of a $C_0$-semigroup $T$ is connected with some notion of regularity of the operators $T(t)$ for small $t>0$ relating to the location of the spectrum of these operators. It remains open to what extent this is indeed the case.

\section*{Acknowledgements}

The author would like to express his thanks to Professor C.J.K.\ Batty for his guidance and an observation that led to Proposition~\ref{delta}, and to the anonymous referee for his or her helpful remarks.

\bibliographystyle{plain}

\end{document}